\documentclass[10pt,3p,times]{elsarticle}
\usepackage{amssymb,amsmath,amsthm,hyperref,cleveref,bm}
\journal{Physica D}
\newtheorem{conjecture}{Conjecture}[section]

\newtheorem{theorem}{Theorem}[section]

\crefname{equation}{}{}
\crefname{figure}{Fig.}{}
\crefname{section}{Section}{}
\begin{document}
\begin{frontmatter}
\title{Solution landscape of the Onsager model identifies non-axisymmetric critical points}
\author[inst1]{Jianyuan Yin}
\ead{yinjy@pku.edu.cn}
\address[inst1]{School of Mathematical Sciences, Laboratory of Mathematics and Applied Mathematics, Peking University, Beijing 100871, China}
\author[inst2]{Lei Zhang\corref{cor1}}\ead{zhangl@math.pku.edu.cn}
\author[inst1]{Pingwen Zhang}\ead{pzhang@pku.edu.cn}
\cortext[cor1]{Corresponding Authors}
\address[inst2]{Beijing International Center for Mathematical Research, Center for Quantitative Biology, Peking University, Beijing 100871, China}
\begin{abstract}
We investigate critical points of the Onsager free-energy model on a sphere with different potential kernels, including the dipolar potential, the Maier--Saupe potential, the coupled dipolar/Maier--Saupe potential, and the Onsager potential.
A uniform sampling method is implemented for the discretization of the Onsager model, and solution landscapes of the Onsager model are constructed using saddle dynamics coupled with downward/upward search algorithms.
We first construct the solution landscapes with the dipolar and Maier--Saupe potentials, for which all critical points are axisymmetric.
For the coupled dipolar/Maier--Saupe potential, the solution landscape shows a novel non-axisymmetric critical point, named \emph{tennis}, which exists for a wide range of parameters.
We further demonstrate various non-axisymmetric critical points in the Onsager model with the Onsager potential, including \emph{square, hexagon, octahedral, cubic, quarter, icosahedral}, and \emph{dodecahedral} states.
The bifurcation diagram is presented to show the primary and secondary bifurcations of the isotropic state and reveal the emergence of the critical points.
The solution landscape provides an efficient approach to show the global structure of the model system as well as the bifurcations of critical points, which can not only support the previous theoretical conjectures but also propose new conjectures based on the numerical findings.
\end{abstract}
\end{frontmatter}

\section{Introduction}
In 1949, Onsager proposed a free-energy model to describe the isotropic-nematic phase transition in three-dimensional rod-like liquid crystals, and related the equilibrium states to critical points of the model \cite{onsager1949effects}.
Since then, this molecular model has been applied to describe static and dynamic phenomena of liquid crystals \cite{doi1988theory,de1995physics,zhang2007stable,wang2021modeling}.
By including the positional dependence of the distribution function, the Onsager model can describe complicated phenomena of liquid crystals \cite{holyst1988director,zhang2012onsager,liang2014rigid,yao2018topological}.

Let $\rho: S^2\to \mathbb{R}$ be a probability density function satisfying,
\begin{equation}\label{eqn:onsagerrho}
   \rho(\mathbf{p})\geqslant0, \quad \int_{S^2} \rho(\mathbf{p}) \mathrm{d}\mathbf{p} = 1,
\end{equation}
where $\rho(\mathbf{p})$ characterizes the probability of finding molecules with orientation $\mathbf{p}$ on the unit sphere $S^2$.
Based on the second virial approximation, the Onsager free-energy model is derived as,
\begin{equation}\label{eqn:onsager}
  E(\rho) = \tau\int_{S^2} \rho(\mathbf{p}) \log \rho (\mathbf{p}) \mathrm{d}\mathbf{p}
  + \frac12 \int_{S^2} \int_{S^2} k(\mathbf{p} \cdot\mathbf{q}) \rho(\mathbf{p}) \rho (\mathbf{q}) \mathrm{d}\mathbf{p} \mathrm{d}\mathbf{q},
\end{equation}
combined with the constraint \cref{eqn:onsagerrho}.
The first term in \cref{eqn:onsager} is an entropy term, and the positive parameter $\tau$ is proportional to the absolute temperature of the system \cite{ball2017liquid}.
The second term in \cref{eqn:onsager} represents the potential energy due to interactions between pairs of molecules but also contains an entropic component \cite{ball2021axisymmetry}.
For the derivation and validity of the Onsager model, one can also refer to \cite{palffy2017onsager}.
The kernel $k$ depends only on the scalar product $\mathbf{p}\cdot\mathbf{q}$ of two orientations $\mathbf{p}$ and $\mathbf{q}$ to ensure the rotational symmetry \cite{vollmer2017critical}.
Important examples of the kernel potentials include,
\begin{enumerate}
  \item Dipolar: $k(t) = -\sigma t$;
  \item Maier--Saupe: $k(t) = -\kappa t^2$;
  \item Coupled dipolar/Maier--Saupe: $k(t) = -\sigma t -\kappa t^2$;
  \item Onsager: $k(t) = \mu \sqrt{1-t^2}$.
\end{enumerate}
Here $\sigma, \kappa, \mu$ are positive parameters \cite{ball2021axisymmetry}.
The Onsager potential, $k(\mathbf{p}\cdot\mathbf{q})=|\mathbf{p}\times\mathbf{q}|$, is one of the first interaction kernels that have been derived, while the non-analytic nature of this potential brings difficulties to explicit non-trivial solutions.
Subsequently, Maier and Saupe \cite{maier1958einfache} suggested using a similar but analytic potential of $k(\mathbf{p}\cdot\mathbf{q})=\frac{1}{3}-|\mathbf{p}\times\mathbf{q}|^2$.
Note that critical points of the Onsager model are not affected by adding a constant to $k$.
For the Maier--Saupe and Onsager potentials, $k$ is an even function, corresponding to the head-to-tail symmetry of liquid-crystal molecules.

For an absolutely continuous interaction kernel $k$, the critical points of the Onsager model are proved to be smooth, bounded away from zero, and satisfy the Euler--Lagrange equation \cite{ball2021axisymmetry}.
A probability density function $\rho$ is called \emph{axisymmetric} if there exists some vector $\mathbf e$ and some function $f$ such that $\rho(\mathbf p)=f(\mathbf p\cdot\mathbf e)$.
For the dipolar potential, any critical point is axisymmetric, and a full description of critical points is presented in \cite{fatkullin2005critical}.
For the Maier--Saupe potential, a complete classification of all the critical points have been studied extensively with different methods in \cite{fatkullin2005critical} and \cite{liu2005axial}, which accords with previous numerical simulations in \cite{faraoni1999rigid,forest2004weak}.
For the coupled dipolar/Maier--Saupe potential, stable critical points are proved to be axisymmetric \cite{zhou2007characterization}.
Furthermore, with some parameters, there exists an unstable critical point that is non-axisymmetric \cite{zhou2007characterization}.
For the Onsager potential, finding and classifying all critical points of the Onsager free-energy model in three dimensions becomes much more complicated.
By reformulating the critical points as an eigenvalue problem, an iterative scheme was derived to compute axisymmetric critical points in \cite{kayser1978bifurcation}.
The complete set of all bifurcation points corresponding to the primary bifurcation of the isotropic state $\rho_0(\mathbf{p})=1/4\pi$ has been obtained in \cite{vollmer2017critical}.
A transcritical bifurcation from this trivial solution occurs first and all critical points around this bifurcation point are axisymmetric.
Recently, Ball proved that non-axisymmetric critical points exist for a wide variety of kernel potentials, including the Onsager potential \cite{ball2021axisymmetry}.
To be specific, for large enough $\mu$, there exists a non-axisymmetric critical point with cubic symmetry, that is, in some Cartesian coordinate $(x,y,z)$,
\begin{equation}\label{eqn:p48}
\rho(x,y,z)=\rho(y,x,z)=\rho(x,z,y)=\rho(-x,y,z).
\end{equation}

Despite the extensive analytical works have been done for the critical points of the Onsager model, how to find these critical points, especially for non-axisymmetric critical points, is an important but unsolved problem.
According to the Morse theory \cite{milnor1963morse}, the (Morse) index of a critical point is the maximal dimension of a subspace on which the Hessian operator is negative definite, and local properties of a critical point can be described by its index.
Compared to the stable (index-0) critical points (\emph{i.e.} minimizers), the unstable (index$>$0) critical points (\emph{i.e.} saddle points) are much more difficult to compute because of their unstable nature.
There are several numerical methods for computing critical points, including the generalized gentlest ascent dynamics \cite{quapp2014locating}, the iterative minimization formulation \cite{gao2015iterative}, and the minimax method \cite{li2001minimax,li2019local2}.
An alternative method is to solve the Euler--Lagrange equation directly, such as homotopy methods \cite{chen2004search,hao2014bootstrapping} and Newton--Raphson methods with a deflation technique \cite{farrell2015deflation},
However, the most existing numerical methods highly rely on good initial guesses for finding multiple critical points, which often requires \emph{a priori} knowledge of the model systems.
Recently, Yin et al. proposed an efficient numerical method to compute the \emph{solution landscape}, which is a pathway map consisting of all critical points and their connections, in order to systematically find multiple critical points without tuning initial guesses \cite{yin2020construction}.
By applying the high-index saddle dynamics (HiSD) method \cite{yin2019high}, the solution landscape can be efficiently constructed by downward/upward search algorithms.
This numerical approach has been successfully applied to the Landau-type free-energy functionals, including the defect landscape of confined nematic liquid crystals using a Landau--de Gennes model \cite{han2021solution,yin2020construction} and the nucleation of quasicrystals from period crystals using the Lifshitz--Petrich model \cite{yin2020nucleation}.

In this paper, we aim to find global structure of the critical points of the Onsager model with different potential kernels by using a solution landscape approach.
In particular, we will identify non-axisymmetric critical points for the coupled dipolar/Maier--Saupe potential and the Onsager potential, and then investigate the bifurcation diagrams to show the emergence of the critical points.
The rest of the paper is organized as follows.
The Onsager model and numerical methods are introduced in \cref{sec:methods}.
Numerical results are presented in \cref{sec:results}.
In order to reveal various non-axisymmetric critical points, we construct the solution landscapes of the Onsager model.
Some conjectures are proposed based on the numerical results.
Final conclusions and discussions are presented in \cref{sec:conclusions}.

\section{Models and methods}\label{sec:methods}
\subsection{Onsager model}
For the Onsager free-energy model \cref{eqn:onsager} with the constraint \cref{eqn:onsagerrho}, we set $\tau=1$ due to simple rescaling and adjust the model with the parameters in the interaction kernel $k$.
The variation of $E$ is,
\begin{equation}\label{eqn:onsagerfv}
  \frac{\delta E}{\delta \rho}(\rho)
  =\log\rho(\mathbf{p}) + 1 + \int_{S^2} k(\mathbf{p} \cdot\mathbf{q}) \rho (\mathbf{q}) \mathrm{d}\mathbf{q}=\log\rho(\mathbf{p}) +1+ \mathcal{K}\rho(\mathbf{p}),
\end{equation}
where the self-adjoint linear operator $\mathcal{K}: L^2(S^2)\to L^2(S^2)$ of the nonlocal interaction is defined as,
\begin{equation}\label{eqn:onsagerkernelconv}
  \mathcal{K}\rho(\mathbf{p}) = \int_{S^2} k(\mathbf{p} \cdot\mathbf{q}) \rho (\mathbf{q}) \mathrm{d}\mathbf{q}.
\end{equation}
To deal with the constraint \cref{eqn:onsagerrho}, the orthogonal projection operator to the subspace $L_0^2(S^2) = \{\rho \in L^2(S^2): \int_{S^2} \rho(\mathbf{p}) \mathrm{d}\mathbf{p}=0\}$ is defined as,
\begin{equation}\label{eqn:onsagerprojectionoperator}
  \mathcal{P}\rho(\mathbf{p})
  = \rho(\mathbf{p}) - \frac{1}{4\pi}\int_{S^2} \rho (\mathbf{q}) \mathrm{d}\mathbf{q},
\end{equation}
and then the derivative of $E$ is $\nabla E(\rho) = \mathcal{P}(\log\rho + \mathcal{K}\rho)$.
The critical points of \cref{eqn:onsager} are solutions to the Euler--Lagrange equation, $DE(\rho)=0$, that is,
\begin{equation}\label{eqn:onsagerel}
  \log\rho(\mathbf{p}) + \int_{S^2} k(\mathbf{p} \cdot\mathbf{q}) \rho (\mathbf{q}) \mathrm{d}\mathbf{q} - \xi=0,
\end{equation}
where $\xi$ is the Lagrange multiplier corresponding to the mass-conservation constraint \cref{eqn:onsagerrho}.
The second variation of $E$ is, for $\upsilon \in L_0^2(S^2)$,
\begin{equation}\label{eqn:onsager2v}
  \frac{\delta^2 E}{\delta \rho^2}(\rho)[\upsilon]=
  \int_{S^2} \left[\frac{\upsilon^2(\mathbf{p})}{\rho(\mathbf{p})} + \upsilon(\mathbf{p}) \mathcal{K} \upsilon(\mathbf{p}) \right]\mathrm{d}\mathbf{p}
  = \langle \upsilon, \upsilon/\rho + \mathcal{K} \upsilon\rangle,
\end{equation}
and the Hessian can be expressed as $\nabla^2 E(\rho) = \mathcal{P} (\rho^{-1} + \mathcal{K})\mathcal{P}$.

From the Euler--Lagrange equation \cref{eqn:onsagerel}, the isotropic state $\rho_0(\mathbf{p}) = 1/4\pi$ is a critical point for general kernels.
The Hessian at the isotropic state, $\nabla^2 E(\rho_0)=4\pi \mathcal{I} +\mathcal{K}$ has the same eigenvectors as those of $\mathcal{K}$.
Since $\mathcal{K}$ has an obvious eigenvector $\rho_0$, all the other eigenvectors are in $L_0^2(S^2)$.
For $\max k(\cdot) \leqslant W(2/\pi)/16\approx 0.0262$, the isotropic state is the unique solution and the global minimizer of the Onsager model \cref{eqn:onsager} \cite{vollmer2017critical}.
Here $W$ denotes the Lambert-W function, the inverse function of $x \exp x$.

Because of the symmetry of the Onsager model \cref{eqn:onsager},
\begin{equation}\label{eqn:rotationalsymmetry}
  E\left(\rho(\mathbf{p})\right)=E\left(\rho\left(\mathbf{R}\mathbf{p}\right)\right), \quad \forall \mathbf{R}\in O(3),
\end{equation}
each anisotropic critical points is on a solution manifold and not isolated.
Therefore, the Hessian \cref{eqn:onsager2v} at an anisotropic critical point have zero eigenvalues.
To be specific, Hessians at axisymmetric anisotropic critical points have zero eigenvalues of multiplicity two, and Hessians at non-axisymmetric ones have zero eigenvalues of multiplicity three.
On the contrary, the Hessian at the isotropic state has no zero eigenvalues in general.
These zero eigenvalues will bring some difficulties to numerical computations.

For the coupled dipolar/Maier--Saupe potential, all the critical points of the Onsager model \cref{eqn:onsager} can be described as critical points of a finite-dimensional model within orthogonal transformations.
This finite-dimensional model can reduce computational costs and, more importantly, Hessians at critical points have no zero eigenvalues in general.
However, critical points may have lower indices in this reduced model, and there may exist extraneous critical points as well.
Therefore, we only present the details of the reduced model in \ref{apd:a} and will not apply it in practical computations.

\subsection{Discretization methods}
Since the Onsager model \cref{eqn:onsager} involves no gradient terms, $N$ points are sampled on the spherical surface $S^2$ to discretize the probability density function $\rho$.
Starting from a regular icosahedron inscribed in the unit sphere, each face is subdivided four times to obtain $N=2562$ points $\{\mathbf{x}_i\}_{i=1}^{N}$ on the sphere.
By minimizing the Thomson problem of $N$ particles,
\begin{equation}\label{eqn:onsagerthomson}
  \min_{\mathbf{x}_1,\cdots, \mathbf{x}_N\in S^2} \sum_{i=1}^N \sum_{j=i+1}^{N}\frac{1}{\|\mathbf{x}_i-\mathbf{x}_j\|_2},
\end{equation}
we obtain $N$ particles $\{\mathbf{x}_i\}_{i=1}^{N}$ approximatively uniformly distributed on the spherical surface.
We have checked the result of this subdivision with an $N=10242$ subdivision and ensured that the $N=2562$ discretization can provide enough accuracy.

We denote $\bm \rho=(\rho_i)=(\rho(\mathbf{x}_i))\in \mathbb{R}^N$ as the node values of the density distribution, and the constraint \cref{eqn:onsagerrho} is discretized as,
\begin{equation}\label{eqn:onsagerdiscons}
  \bm \rho >0;\quad h\bm 1^\top \bm\rho = 1,
\end{equation}
where $h=4\pi/N$ is the average area per particle, and $\bm 1\in \mathbb{R}^N$ is the vector with all entries one.
The Onsager free-energy model is discretized as,
\begin{equation}\label{eqn:onsagerdisenergy}
  E(\bm\rho)
  =h\sum_{i=1}^{N}\rho_i \log \rho_i + \frac{h^2}{2} \sum_{i=1}^{N}\sum_{j=1}^{N} K_{ij}\rho_i \rho_j
  = h \bm \rho^\top \log\bm \rho +
  \frac{h^2}{2} \bm \rho^\top \mathbf{K} \bm \rho,
\end{equation}
where $\mathbf{K} = (K_{ij})=(k(\mathbf{x}_i \cdot \mathbf{x}_j))$.

As the interaction term becomes dominant, some equilibrium density distributions can be highly anisotropic, and many entries of $\bm\rho$ can be extremely close to zero which may suffer from underflow.
Therefore, we iterate $\log\bm\rho$ instead of $\bm \rho$ in practical computations.
For the almost-uniformly-sampled points $\{\mathbf{x}_i\}$, we simply consider the Euclidian inner product instead of the discretized $L^2$ inner product.
The gradient flow of \cref{eqn:onsagerdisenergy} is,
\begin{equation}\label{eqn:onsagerdisgf}
  \dot{\bm \rho} = -\nabla E(\bm \rho)=-h\log \bm \rho - h^2 \mathbf K \bm \rho + \xi \bm{1},
\end{equation}
where $\xi$ is the Lagrangian multiplier.
Since $\bm \rho$ can be extremely close to zero, an explicit Euler scheme often requires a tiny step size due to a large $|\log \bm \rho|$.
Therefore, we impose a semi-implicit scheme with a fixed step size $\Delta t$ by taking the nonlinear part implicitly,
\begin{equation}\label{eqn:onsagergfsi}
  \exp \bm \psi^{n+1} +\Delta t \;h\bm \psi^{n+1} = \bm \rho^n -\Delta t \; h^2\mathbf K \bm \rho^n +\Delta t\; \xi^n \bm{1},
\end{equation}
where the multiplier $\xi^n$ is calculated as,
\begin{equation}\label{eqn:onsagergfxi}
  \xi^n = N^{-1} \bm 1^\top (h\log \bm{\rho}^n+ h^2 \mathbf{K} \bm{\rho}^n).
\end{equation}
Here, $\bm{\rho}^n$ represents the numerical solution $\bm\rho$ at the $n$-th iteration step, and $\bm \psi^{n+1}$ represents the numerical solution $\log\bm\rho$ at the $(n+1)$-th iteration step which may not satisfy the constraint \cref{eqn:onsagerdiscons}.
Although  $\bm \psi^{n+1}$ can be directly solved from \cref{eqn:onsagergfsi} using the Lambert-W function, this procedure can be numerically unstable for small $\Delta t$.
Alternatively, we solve $\bm \psi^{n+1}$ from \cref{eqn:onsagergfsi} using the Newton's method with the initial value $\log {\bm \rho}^n$, and calculate
\begin{equation}\label{eqn:onsagerretraction}
  \log {\bm \rho}^{n+1}=
  \bm \psi^{n+1} -\log (h \bm 1^\top  \exp \bm \psi^{n+1})\bm{1},
\end{equation}
to obtain a density $\bm\rho^{n+1}$ satisfying the constraint \cref{eqn:onsagerdiscons}.

\subsection{Saddle dynamics}
The saddle dynamics is designed to search unstable saddle points with a given index.
A nondegenerate index-$k$ saddle point ($k$-saddle) is a local maximum on a $k$-dimensional subspace $\mathcal{V}$, and a local minimum on its orthogonal complement $\mathcal{V}^\perp$.
The HiSD for a $k$-saddle ($k$-HiSD) is given by,
\begin{equation}\label{eqn:onsagerhisd}
\left\{
\begin{aligned}
\dot{\bm{\rho}}  &=-\left(\mathbf{I}-\sum\limits_{i=1}^{k} 2\bm{v}_{i} \bm{v}_{i}^{\top}\right) \nabla {E}(\bm{\rho}), \\
\dot{\bm{v}_i}&=-\left(\mathbf{I}-\bm{v}_i\bm{v}_i^\top- \sum\limits_{j=1}^{i-1} 2\bm{v}_j\bm{v}_j^\top\right) \nabla^2 E(\bm{\rho})\bm{v}_i, \quad i=1, \cdots, k.
\end{aligned}
\right.
\end{equation}
The $k$-HiSD involves a position variable $\bm \rho$ and $k$ ascent variables $\bm v_i$ with an initial condition,
\begin{equation}\label{eqn:onsagerhisdinitial}
\bm \rho=\bm \rho^{(0)}, \quad \bm v_i = \bm v_i^{(0)}\in \mathbb{R}_0^N, \quad
\mathrm{s.t.}\left\langle\bm v_j^{(0)},\bm v_i^{(0)}\right\rangle=\delta_{ij}, \quad i,j=1,\cdots,k,
\end{equation}
where $\mathbb{R}_0^N=\{\bm v \in\mathbb{R}^N: \bm 1^\top \bm v=0\}$ ensures the constraint \cref{eqn:onsagerdiscons}.
The dynamics for $\bm\rho$ in \cref{eqn:onsagerhisd} is a transformed gradient flow, \emph{i.e.},
\begin{equation}\label{eqn:onsagerhisdtransform}
\dot{\bm{\rho}} = \mathbf{P}_{\mathcal{V}}\nabla E(\bm \rho)- (\mathbf{I}-\mathbf{P}_{\mathcal{V}})\nabla E(\bm \rho).
\end{equation}
Here $\mathbf{P}_{\mathcal{V}}\nabla E(\bm \rho)$, which is the orthogonal projection of $\nabla E(\bm \rho)$ on the subspace $\mathcal{V}$, represents the gradient ascent direction to find a local maximum along the subspace $\mathcal{V}=\operatorname{span}\{\bm v_1, \cdots,\bm v_k\}$, while $(\mathbf{I}-\mathbf{P}_{\mathcal{V}})\nabla E(\bm \rho)$ is the gradient descent direction on $\mathcal{V}^\perp$.

The dynamics for $\bm v_i$ in \cref{eqn:onsagerhisd} renews the subspace $\mathcal{V}$ by finding the eigenvectors corresponding to the smallest $k$ eigenvalues of the Hessian $\nabla^2 E(\bm \rho)$ at the current position $\bm \rho$.
The $i$-th eigenvector $\bm v_i$ is obtained by solving a constrained optimization problem,
\begin{equation}\label{eqn:onsagerhisdrayleigh}
\min_{\bm v_i \in \mathbb{R}_0^N} \quad \left\langle\nabla^2 E(\bm \rho) [\bm v_i], \bm v_i\right\rangle \quad \mathrm{s.t.}\left\langle\bm v_j,\bm v_i\right\rangle=\delta_{ij},\quad j=1,\cdots,i-1,
\end{equation}
with the knowledge of $\bm v_1, \cdots,\bm v_{i-1}$ using the gradient flow.
The dynamics \cref{eqn:onsagerhisd} is linearly stable if and only if at index-$k$ saddle points \cite{yin2019high}.
The HiSD method can also be generalized for equality-constrained cases \cite{yin2020constrained} and non-gradient (dynamical) systems \cite{yin2020searching}.

The Hessian of \cref{eqn:onsagerdisenergy} $\nabla^2 E(\bm \rho)$ is a linear operator on $\mathbb{R}_0^N$,
\begin{equation}\label{eqn:onsagerdishess}
  \nabla^2 E(\bm \rho)=
  \mathbf{P}\left(h\operatorname{Diag}(\bm\rho^{-1})+h^2\mathbf{K}\right)\mathbf{P},
\end{equation}
where $\mathbf{P} = \mathbf{I} - \bm 1\bm 1^\top / N$ is the projection matrix concerning the constraint \cref{eqn:onsagerdiscons}.
Here $\bm\rho^{-1}=(\rho_i^{-1})$ is the elementwise inverse of $\bm \rho$.
Because of the high anisotropy of $\bm \rho$, Hessians \cref{eqn:onsagerdishess} at critical points can be highly ill-conditioned.
Fortunately, in the saddle dynamics, only the smallest $k$ eigenvalues need to be considered, so we alternatively replace $\nabla^2 E(\bm\rho)$ in \cref{eqn:onsagerhisd} with,
\begin{equation}\label{eqn:onsagerdishessm}
  H_M(\bm \rho)=
  \mathbf{P}\left(h\operatorname{Diag}(\max\{\bm\rho^{-1}, M\})+h^2\mathbf{K}\right)\mathbf{P},
\end{equation}
where $M$ is a large constant set as $10^8$.
In fact, for a large $M$ from $10^6$ to $10^{10}$, $H_M(\bm \rho)$ can present the smallest several eigenvalues and corresponding eigenvectors which are fairly consistent with $\nabla^2 E(\bm \rho)$ numerically.
In numerical implementation, the dynamics of $\bm v_i$ is solved using the locally optimal block preconditioned conjugate gradient (LOBPCG) method \cite{knyazev2001toward}, and the details can be found in \cite{yin2019high}.

For numerical stability, the saddle dynamics is solved in a similar manner to the gradient flow \cref{eqn:onsagerdisgf} by taking the nonlinear part implicitly.
The time discretization of $\bm \rho$ in the HiSD \cref{eqn:onsagerhisd} is given by,
\begin{equation}\label{eqn:onsagerhisdrho}
\frac{\exp{\bm \psi}^{n+1} - \bm \rho^n}{\Delta t}  =-h{\bm \psi}^{n+1} - h^2 \mathbf K \bm \rho^n + \xi^n \bm{1}+\sum\limits_{i=1}^{k} 2 \langle \bm{v}_{i}^n, \nabla {E}(\bm{\rho}^n)\rangle\bm{v}_{i}^n,
\end{equation}
where $\xi^n$ is calculated as \cref{eqn:onsagergfxi}.
Then ${\bm \psi}^{n+1}$ is solved from \cref{eqn:onsagergfsi} using Newton's method with the initial value $\log {\bm \rho}^n$, and ${\bm \rho}^{n+1}$ is obtained by \cref{eqn:onsagerretraction} eventually.
This iteration is terminated if $\|\nabla E(\bm \rho^n)\|$ is smaller than the tolerance.

\subsection{Algorithm of solution landscape}

The procedure for construction of the solution landscape consists of a downward search algorithm and an upward search algorithm.
Starting from a parent state that is a high-index saddle point, the downward search algorithm is applied to search critical points with lower indices by following its unstable directions.
Assume $\hat{\bm \rho}$ is an index-$m$ saddle point and $\hat{\bm u}_1, \cdots, \hat{\bm u}_m$ are the $m$ orthonormal eigenvectors of $\nabla^2 E(\hat{\bm\rho})$, corresponding to the smallest $m$ eigenvalues $\hat{\lambda}_1< \cdots< \hat{\lambda}_m<0$, respectively.
Then we attempt to search critical points of indices from 0 to $m-1$.
To find an index-$k$ critical point, we first slightly perturb $\hat{\bm\rho}$ along an unstable direction (typically $\hat{\bm u}_{k+1}$).
Then the $k$-HiSD is solved from the initial position $\bm \rho(0)=\hat{\bm \rho}+\varepsilon \hat{\bm u}_{k+1}$ and initial directions $\bm v_i(0)=\hat{\bm u}_i$ for $i=1,\cdots,k$.
Because of the degeneracy of critical points, this $k$-HiSD may converge to a critical point whose index is smaller than $k$, which should be checked after convergence.
An arrow from $\hat{\bm \rho}$ to this critical point is drawn in the solution landscape to illustrate this relationship.
Then this downward search procedure is repeated to the newly-found saddle points, until no more critical points are found.

If the parent state is unavailable, the upward search algorithm can be used to find a high-index critical point starting from a low-index saddle point or a minimum.
Assume $\hat{\bm \rho}$ is an index-$m$ saddle point with $k$ zero eigenvalues, and $\hat{\bm u}_1, \cdots, \hat{\bm u}_{m+k+1}$ are the $(m+k+1)$ orthonormal eigenvectors of $\nabla^2 E(\hat{\bm\rho})$, corresponding to the smallest $(m+k+1)$ eigenvalues $\hat{\lambda}_1< \cdots< \hat{\lambda}_m<0=\hat{\lambda}_{m+1}=\cdots = \hat{\lambda}_{m+k}< \hat{\lambda}_{m+k+1}$, respectively.
To find a saddle point with a higher index, we perturb $\hat{\bm \rho}$ along a stable direction (e.g. $\hat{\bm u}_{m+k+1}$).
Then the $(m+k+1)$-HiSD is solved from the initial position $\bm \rho(0)=\hat{\bm \rho}+\varepsilon \hat{\bm u}_{m+k+1}$ and the initial directions $\bm v_i(0)=\hat{\bm u}_i$ for $i=1,\cdots,m+k+1$.
If the HiSD converges to a high-index saddle point, we can repeat the upward search algorithm starting from this saddle till no higher-index saddle point can be found.
Once the parent state is found, the downward search from it can obtain all connected critical points in the solution landscape.

For the Onsager model \cref{eqn:onsager}, we implement both downward and upward searches from the isotropic state $\bm \rho_0=\bm 1/4\pi$.
The downward search guarantees the systematic finding of critical points connected by the isotropic state.
The upward search from the isotropic state is applied to find possible critical points with higher indices because of saddle-node bifurcations while tracing only the isotropic state could fail to locate such saddle points, for example, the prolate$'$ state for the Maier--Saupe potential at $\kappa=7$ and for the Onsager potential at $\mu=10$ in the next section.

\section{Solution landscape of the Onsager model}\label{sec:results}

\subsection{Dipolar potential}

The simplest kernel of the Onsager model is the dipolar potential $k(t) = -\sigma t$.
$\mathcal{K}$ has an only negative eigenvalue $-\frac{4}{3}\pi\sigma$, with three linearly independent eigenvectors $v(\mathbf{p})=p_i$, ($i=1,2,3$);
all the other eigenvalues are zero (with infinite multiplicity).
Therefore, the isotropic state $\rho_0$ is a minimizer for $\sigma<3$, and a 3-saddle for $\sigma>3$.
From the Euler--Lagrange equation \cref{eqn:onsagerel}, a critical point $\rho$ satisfies,
\begin{equation}\label{eqn:dipolarsp}
  \rho(\mathbf{p}) = \exp\left(\xi + \int_{S^2} -\sigma \mathbf{p}\cdot\mathbf{q} \rho(\mathbf{q})\mathrm{d}\mathbf{q}\right)
  =\mathrm{e}^\xi\exp\left(-\sigma\mathbf{p}\cdot\int_{S^2}\mathbf{q} \rho(\mathbf{q})\mathrm{d}\mathbf{q}\right),
\end{equation}
so all critical points are axisymmetric for the dipolar potential.
In fact, critical points other than the isotropic state are proved to be,
\begin{equation}\label{eqn:dipolarstat}
  \rho(\mathbf{p}) = \frac{1}{Z} \exp(-r \mathbf{p}\cdot\mathbf{m}), \quad \sigma = \frac{r^2}{r \coth r -1},
\end{equation}
where $\mathbf{m}$ is a unit vector and $Z=\int_{S^2}\exp(-r \mathbf{p}\cdot\mathbf{m}) \mathrm{d}\mathbf{p}$ denotes the partition function \cite{fatkullin2005critical}.
This nontrivial critical point exists for $\sigma>3$ as a minimum, which emerges via a pitchfork bifurcation from the isotropic state $\rho_0$, and is referred to as a \emph{nematic} state.

\begin{figure*}[htbp]
\centering
\includegraphics[width=\linewidth]{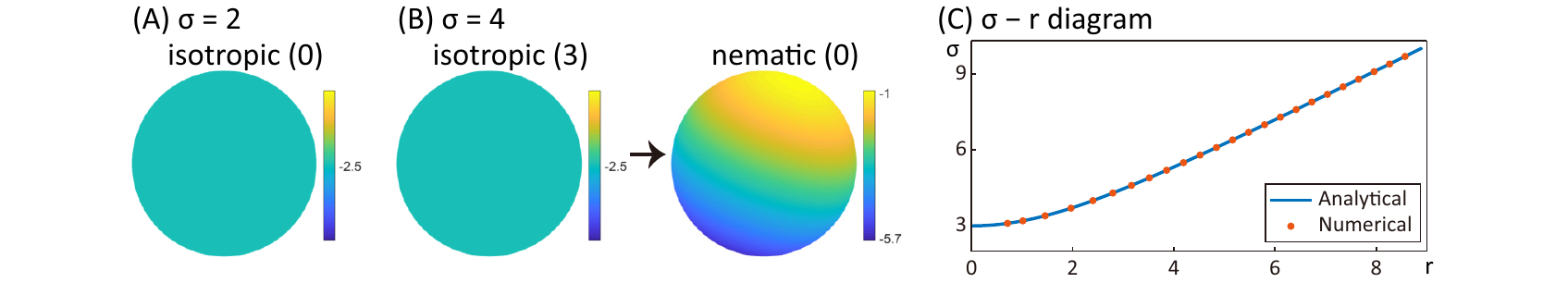}
\caption{Solution landscapes for the dipolar potential at (A) $\sigma=2$, and (B) 4.
(C) Comparison between the analytical and numerical results of $r$ at different $\sigma$ for the dipolar potential.
The critical points in all Figures are plotted with $\log \rho$ and their names are labelled above with the indices in brackets ($\log (1/4\pi) \approx -2.5$).
}
\label{fig:ons1}
\end{figure*}

The solution landscapes at $\sigma=2$ and $4$ are shown in \cref{fig:ons1}(A--B), and the critical points in all Figures are shown with a logarithmic scale for better illustration.
Compared with the analytical results in \cref{eqn:dipolarstat}, the numerical nematic state at $\sigma=4$ has an $L^2$-error of 1.2e-5.
Furthermore, we calculate the nematic states at different $\sigma$ and compare the numerical range of $\frac{1}{2}\log \rho$ with the analytical solution, as shown in \cref{fig:ons1}(C).
The discretization method provides fine accuracy for numerical results.

\subsection{Maier--Saupe potential}

\begin{figure*}[htbp]
\centering
\includegraphics[width=\linewidth]{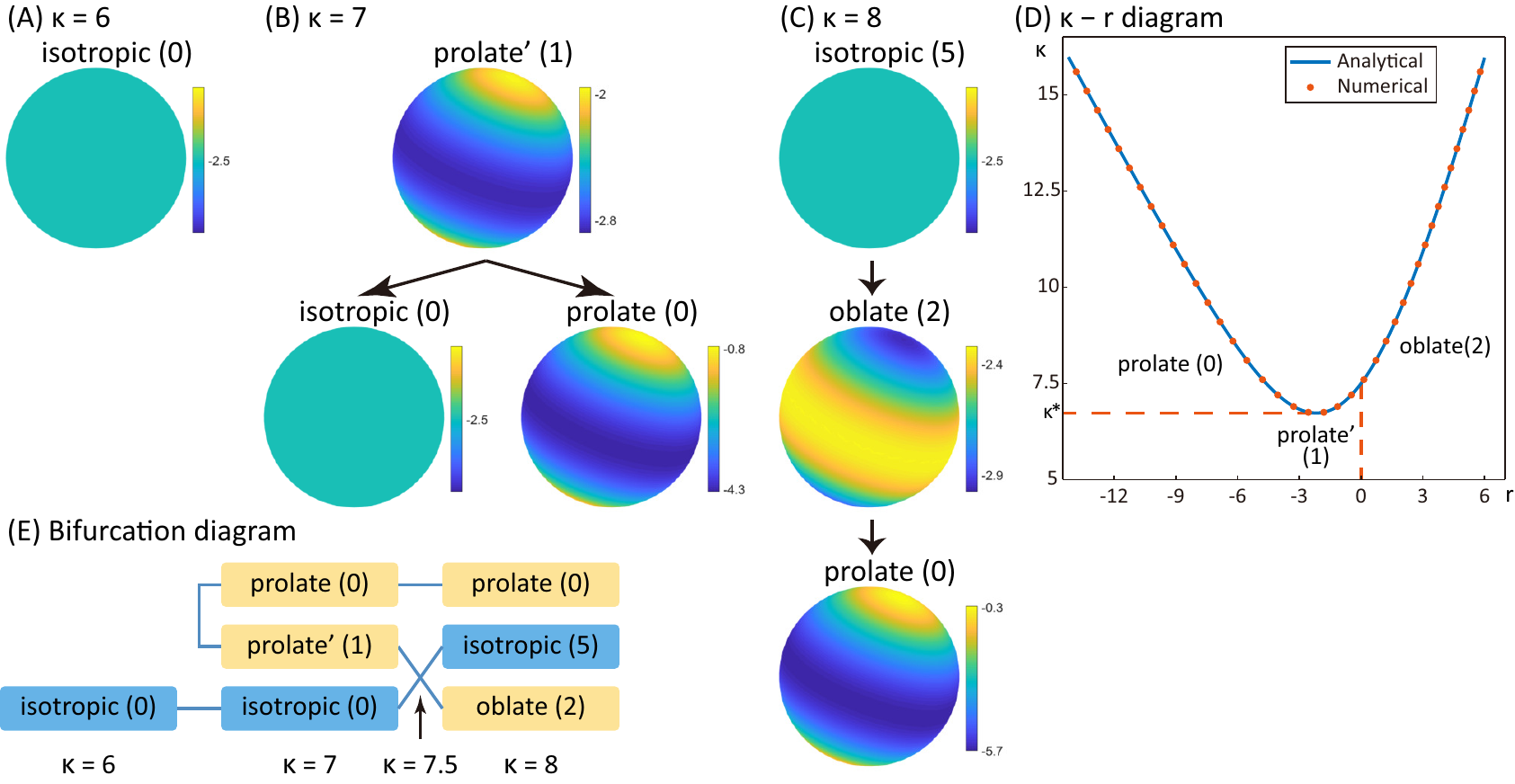}
\caption{Solution landscapes for the Maier--Saupe potential at (A) $\kappa=6$, (B) 7, and (C) 8.
(D) Comparison between the analytical and numerical results of $r$ at different $\kappa$ for the Maier--Saupe potential.
The curve has a tangent line $\kappa=\kappa^*$ separating the prolate and prolate$'$ states, and passes through the bifurcation point $(0,7.5)$ separating the prolate$'$ and oblate states.
(E) The bifurcation diagram for the Maier--Saupe potential.
The index of each state is in the bracket, and the corresponding parameter $\kappa$ is labelled below.
The bifurcation point of the isotropic state is $\kappa=7.5$.
}
\label{fig:ons2}
\end{figure*}

For the Maier--Saupe potential $k(t) = -\kappa t^2$, $\mathcal{K}$ has an only negative eigenvalue $-\frac{8}{15}\pi\kappa$, with five linear independent eigenvectors $v(\mathbf{p})=p_i^2-1/3$ and $v(\mathbf{p})=p_i p_j$.
All the other eigenvalues are zero (with infinite multiplicity).
The solution landscapes for the Maier--Saupe potential are constructed at $\kappa=6,7$ and $8$ in \cref{fig:ons2}(A--C).
The isotropic state $\rho_0$ is a minimizer for $\kappa<7.5$, and a 5-saddle for $\kappa>7.5$.
Moreover, the critical points other than the isotropic state $\rho_0(\mathbf{p})$ are proved to be,
\begin{equation}\label{eqn:msstat}
  \rho(\mathbf{p}) = \frac{1}{Z} \exp\left(-r (\mathbf{p}\cdot\mathbf{m})^2 \right), \quad
  \kappa = 4r^2
  \left(2r- 3+
  \frac{6\sqrt{r}\exp (-r)}{\sqrt{\pi}\operatorname{erf}\sqrt{r}}\right)^{-1},
\end{equation}
where $\mathbf{m}$ is a unit vector \cite{liu2005axial,fatkullin2005critical}.
This nontrivial solution is referred to as a \emph{prolate} state with density concentrating at two antipodal points if $r<0$, and an \emph{oblate} state with density concentrating on a great circle if $r>0$.
Two prolate states emerge at $\kappa^\ast\approx 6.7315$ with $r^\ast \approx -2.18$, as a local minimizer and a 1-saddle (referred to as prolate$'$ to distinguish), respectively.
As $\kappa$ increases, the prolate state becomes the global minimizer in replacement of the isotropic state.
After a transcritical bifurcation at $\kappa=7.5$, the isotropic state $\rho_0$ becomes a 5-saddle, while the oblate state appears as a 2-saddle with $r>0$ \cite{fatkullin2005critical}.
The minimum prolate nematic state remains as a global minimizer (\cref{fig:ons2}(C)), and the bifurcation diagram is shown in \cref{fig:ons2}(E).
Compared with the analytical result \cref{eqn:msstat}, the numerical prolate and oblate states at $\kappa=8$ have $L^2$-errors of 7.0e-4 and 4.1e-5.
We further calculate the prolate($'$) and oblate states at different $\kappa$ and compare the numerical range of $\log \rho$ with the analytical solutions in \cref{fig:ons2}(D) to show the accuracy and consistency of numerical methods.

\subsection{Coupled dipolar/Maier--Saupe potential}

\begin{figure*}[htbp]
\centering
\includegraphics[width=\linewidth]{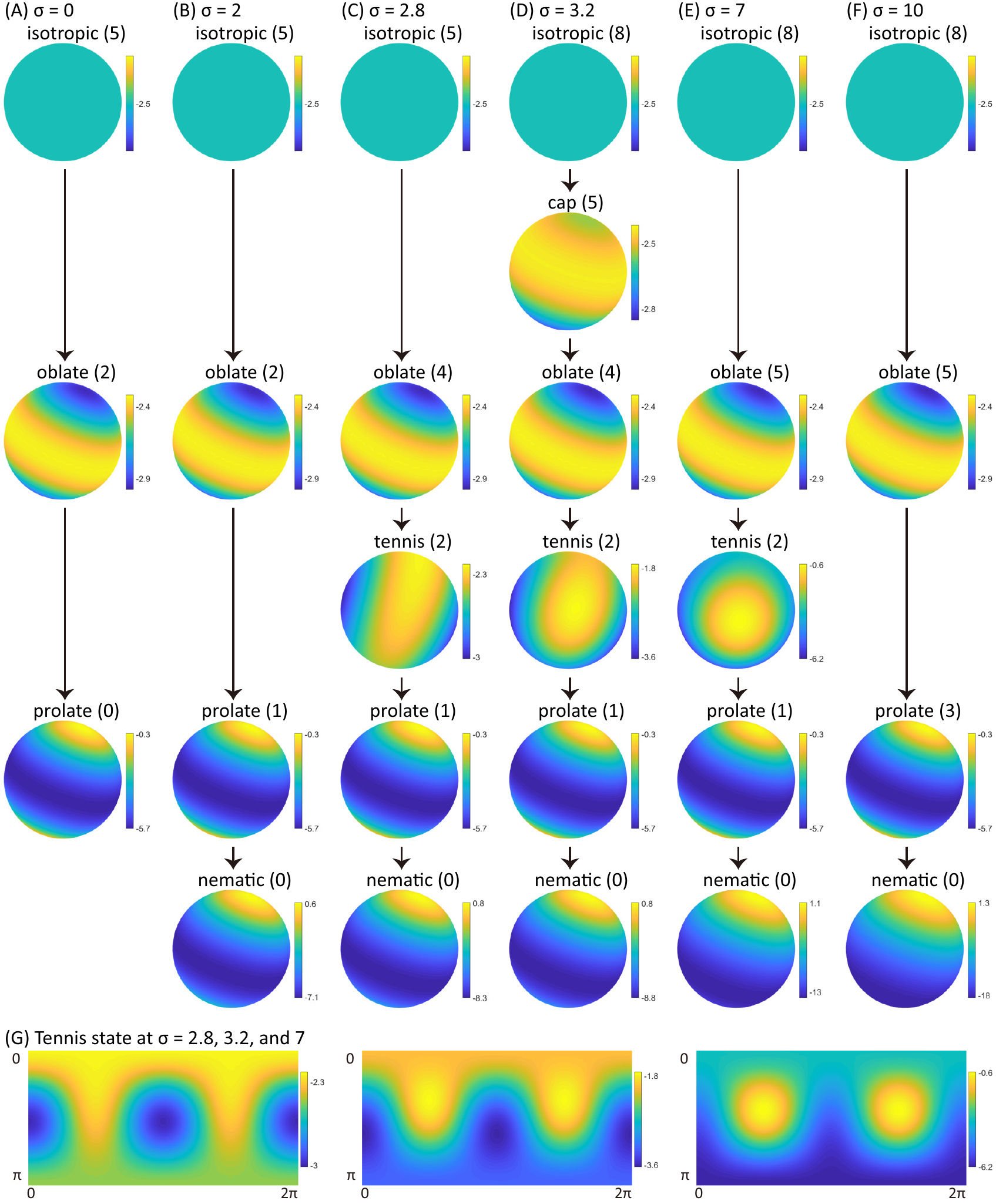}
\caption{Solution landscapes for the coupled dipolar/Maier--Saupe potential at $\kappa=8$ and (A) $\sigma=0$, (B) 2, (C) 2.8, (D) 3.2, (E) 7, and (F) 10.
(G) Expanded two-dimensional plots in spherical coordinates for the tennis state at $\sigma=2.8, 3.2,$ and $7$.
The horizontal axis represents azimuthal angle in $[0,2\pi]$ and the vertical axis represents the polar angle in $[0,\pi]$.
}
\label{fig:ons3}
\end{figure*}

\begin{figure*}[htbp]
\centering
\includegraphics[width=\linewidth]{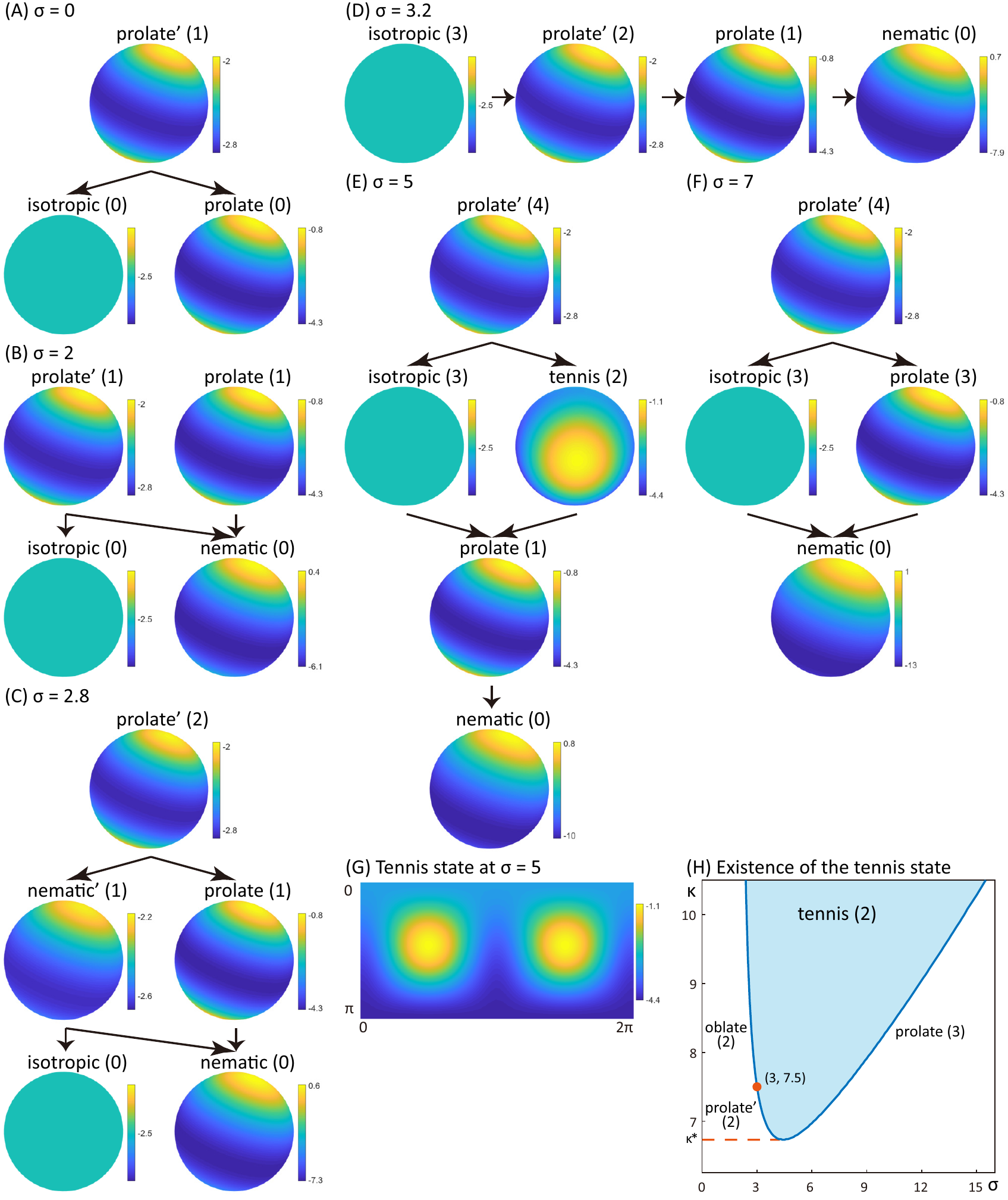}
\caption{Solution landscapes for the coupled dipolar/Maier--Saupe potential at $\kappa=7$ and (A) $\sigma=0$, (B) 2, (C) 2.5, (D) 3.2, (E) 5, and (F) 7.
(G) Expanded two-dimensional plot in spherical coordinates for the tennis state at $\sigma=5$.
(H) Existence region of the non-axisymmetric tennis 2-saddle, with a boundary point $(3, 7.5)$ and a tangent line $\kappa=\kappa^*$ separating bifurcations from the oblate, prolate$'$, and prolate states.
}
\label{fig:ons4}
\end{figure*}

For the coupled dipolar/Maier--Saupe potential $k(t) = -\sigma t -\kappa t^2$, $\mathcal{K}$ has two negative eigenvalue $-\frac{4}{3}\pi\sigma$ and $-\frac{8}{15}\pi\kappa$, and all other eigenvalues are zero (with infinite multiplicity), so the index of the isotropic state $\rho_0$ is not greater than 8.
With simple calculations, it can be seen that for a fixed $\kappa$, critical points at $\sigma=0$ (\emph{i.e.} the Maier--Saupe potential) remain as critical points at $\sigma>0$, while their indices may increase.

We first fix $\kappa=8$ and construct the solution landscapes for the coupled dipolar/Maier--Saupe potential at different $\sigma$ in \cref{fig:ons3}(A--F).
At $\sigma=0$, we recall that there exist a minimum prolate state, a 2-saddle oblate state, and a 5-saddle isotropic state (\cref{fig:ons3}(A)).
At $\sigma=2$, the prolate state becomes a 1-saddle via a pitchfork bifurcation, and a nematic state appears as a global minimizer (\cref{fig:ons3}(B)).
Although this minimizer cannot be expressed in the form of \cref{eqn:dipolarstat}, we refer to it as a nematic state as well, because its density mainly focuses at one pole.
At $\sigma=2.8$, the oblate state becomes a 4-saddle via a pitchfork bifurcation, and a \emph{tennis} state emerges as a 2-saddle with density concentrating around parts of a great circle (\cref{fig:ons3}(C)).
The isotropic state is destabilized as a 8-saddle at $\sigma>3$, and a \emph{cap} state density concentrating on one side emerges as a 5-saddle.
This cap 5-saddle then merges with the oblate 4-saddle, resulting in an oblate 5-saddle.
Finally, the tennis state merges with the prolate state, which leads to a prolate 3-saddle.
The solution landscapes at $\sigma=3.2$, $7$, and $10$ are shown in \cref{fig:ons3}(D--F).
For better illustration, the bifurcation diagram at $\kappa=8$ is shown in \cref{fig:ons8}(A).

Among these anisotropic states, the tennis state is a non-axisymmetric 2-saddle and has the dihedral symmetry $D^2$, that is, in some spherical coordinate $(\theta,\varphi)$ (the azimuthal angle $\theta\in[0,2\pi]$ and the polar angle $\varphi\in[0,\pi]$), the tennis state $\rho$ satisfies,
\begin{equation}\label{eqn:d2}
\rho(\theta,\varphi)=\rho(\theta+\pi,\varphi)=\rho(-\theta,\varphi).
\end{equation}
To illustrate the evolution of the non-axisymmetric state, we plot the tennis state for different $\sigma$ in spherical coordinates in \cref{fig:ons3}(G).
The tennis state emerges from the oblate state, and merges with the prolate state, as shown in \cref{fig:ons8}(A).

Since the solution landscapes are quite different at $\kappa=7$ and $8$ for the Maier--Saupe potential, we further compute the solution landscapes at $\kappa=7$ by varying $\sigma$, as shown in \cref{fig:ons4}(A--F).
Two minima, the prolate state and the isotropic state, are connected by the prolate$'$ 1-saddle at $\sigma=0$, as shown in \cref{fig:ons4}(A).
At $\sigma=2$, a nematic state emerges as a minimizer via a pitchfork bifurcation from the prolate state (\cref{fig:ons4}(B)).
Similarly, a nematic$'$ 1-saddle has been bifurcated from the prolate$'$ state at $\sigma=2.8$ (\cref{fig:ons4}(C)), and then merges into the isotropic state at $\sigma=3$, leading to an isotropic 3-saddle.
At $\sigma=5$, a non-axisymmetric tennis state has emerged as a 2-saddle via a pitchfork bifurcation from the prolate$'$ state, and then merges with the prolate state at $\sigma=7$.
The solution landscapes at $\sigma=3.2$, $5$, and $7$ are shown in \cref{fig:ons4}(D--F), and the tennis state at $\sigma=5$ is plotted in \cref{fig:ons4}(G) in the spherical coordinate.
The bifurcation diagram is shown in \cref{fig:ons8}(B).

\begin{figure*}[htbp]
\centering
\includegraphics[width=\linewidth]{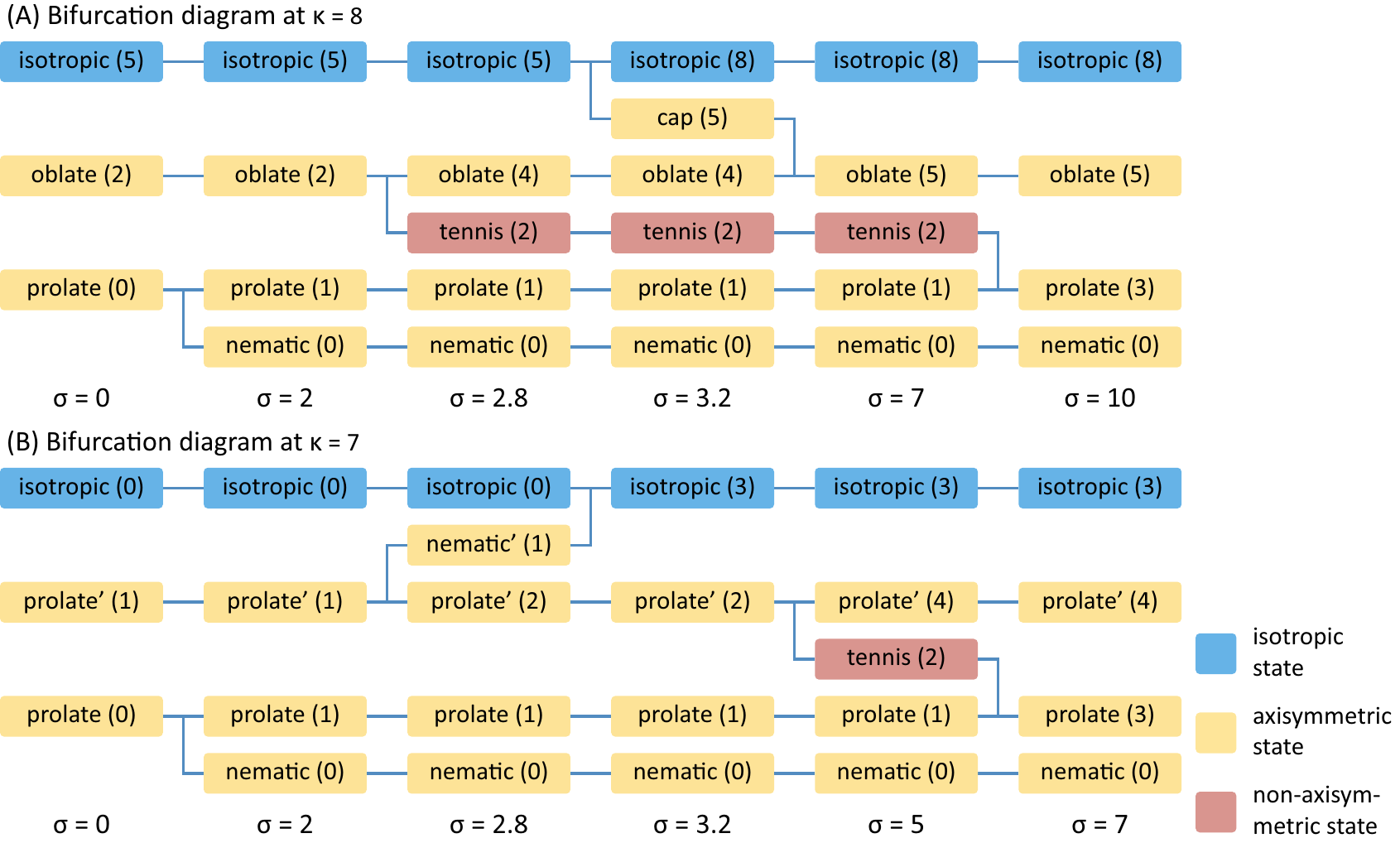}
\caption{
The bifurcation diagrams for the coupled dipolar/Maier--Saupe potential at (A) $\kappa=8$, and (B) $\kappa=7$.
The color of each node indicates its axisymmetric property with its index in the bracket, and the corresponding parameter $\sigma$ is labelled below.
}
\label{fig:ons8}
\end{figure*}

In \cref{fig:ons4}(H), we plot the existence region of the non-axisymmetric tennis state in the parameter space.
The tennis state exists only at $\kappa>\kappa^*$, where $\kappa^*\approx 6.7315$ is the critical $\kappa$ for the existence of anisotropic states for the Maier--Saupe potential.
By increasing $\sigma$, the tennis state is a 2-saddle bifurcated from the prolate$'$ 2-saddle (for $\kappa^*<\kappa<7.5$) or the oblate 2-saddle (for $\kappa>7.5$), and finally merges into the prolate 3-saddle.
The bifurcation point $(\sigma=3, \kappa=7.5)$ of the isotropic state is on the boundary of the existence region of the tennis state.
Based on the numerical findings, we have the following conjectures for the coupled dipolar/Maier--Saupe potential.

\begin{conjecture}
  All the critical points have the dihedral symmetry $D^2$.
\end{conjecture}
\begin{conjecture}
  There exists at most one non-axisymmetric critical point (regardless of the rotation) for each parameter $(\sigma_0, \kappa_0)$, which is an index-2 critical point.
\end{conjecture}
\begin{conjecture}
  If a non-axisymmetric critical point exists for $(\sigma_0, \kappa_0)$, then that also exists for $(\sigma_0, \kappa>\kappa_0)$.
\end{conjecture}

\subsection{Onsager potential}
\begin{figure*}[htbp]
\centering
\includegraphics[width=\linewidth]{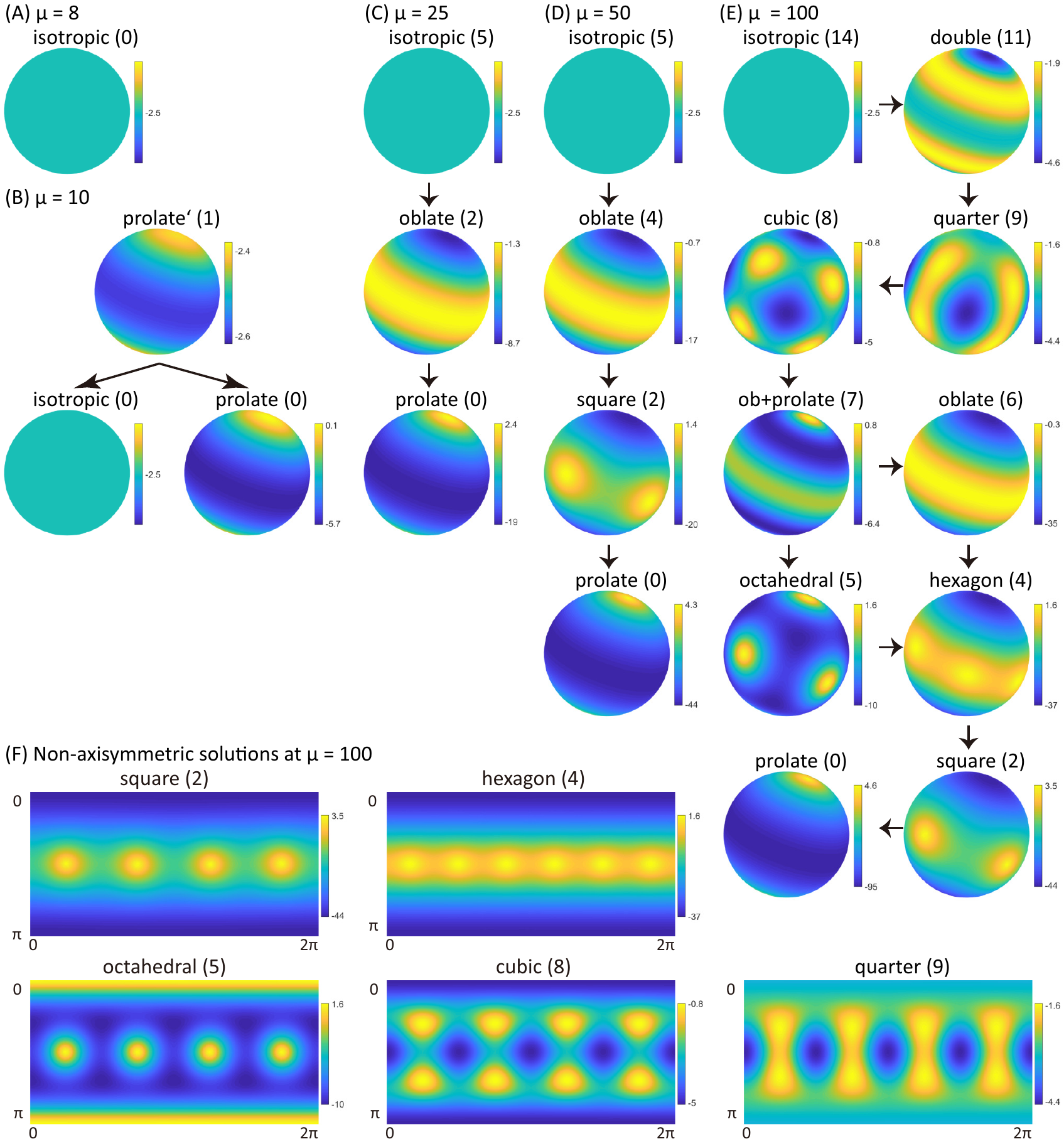}
\caption{Solution landscapes for the Onsager potential at (A) $\mu=8$, (B) 10, (C) 25, (D) 50, and (E) 100.
(F) Expanded two-dimensional plots in spherical coordinates for the non-axisymmetric solutions at $\mu=100$.
The horizontal axis represents azimuthal angle in $[0,2\pi]$ and the vertical axis represents the polar angle in $[0,\pi]$.
}
\label{fig:ons5}
\end{figure*}

For the Onsager potential $k(t) = \mu \sqrt{1-t^2}$, $\mathcal{K}$ has an infinite number of negative eigenvalues as $-\frac{1}{2} \Gamma(s+1)^{-1} \Gamma(s+2)^{-1} \Gamma(s-\frac12) \Gamma(s+\frac12)\pi\mu$ (with multiplicity $4s+1$, $s\in \mathbb{N}_+$) \cite{vollmer2017critical}.
All the other eigenvalues are zero (with infinite multiplicity) with odd functions as the corresponding eigenvectors.
The bifurcation points for the isotropic state $\rho_0$ are $\mu_s = 8\Gamma(s-\frac12)^{-1} \Gamma(s+\frac12)^{-1}  \Gamma(s+1) \Gamma(s+2)$, $s\in \mathbb{N}_+$, and the first three bifurcation points are $\mu_1=32/\pi\approx 10.2$, $\mu_2=256/\pi\approx 81.5$, and $\mu_3=4096/5\pi \approx261$.
As predicted in \cite{ball2021axisymmetry}, non-axisymmetric critical points with cubic symmetry can exist with lower energy than $E(\rho_0)$ if $\mu>\mu_2$.

For sufficiently small $\mu$ (e.g. $\mu=8$), the isotropic state is the only critical point and the global minimizer (\cref{fig:ons5}(A)).
At $\mu=10$, two prolate states have emerged as a minimizer and a 1-saddle (referred to as prolate$'$) respectively via a saddle-node bifurcation, while the isotropic state remains as a minimizer.
The prolate$'$ 1-saddle can be found by upward search from the isotropic state.
After $\mu>\mu_1$, the isotropic state loses its local stability and becomes a 5-saddle, while the 1-saddle prolate$'$ state becomes a 2-saddle oblate state via a transcritical bifurcation, as proved in \cite{vollmer2017critical} and obtained numerically in \cite{kayser1978bifurcation,gopinath2004observations}.
This transcritical bifurcation bears a resemblance to the bifurcation at $\kappa=7.5$ for the Maier--Saupe potential.
It should be noted that these prolate and oblate states are axisymmetric critical points named for their qualitative features and cannot be expressed in the form of \cref{eqn:msstat}.
The solution landscape at $\mu=25$ is shown in \cref{fig:ons5}(C).

At $\mu=50$, the oblate state becomes a 4-saddle via a pitchfork bifurcation, and a \emph{square} state exists as a 2-saddle, with density concentrating at four vertices of a square on a great circle (\cref{fig:ons5}(D)).
The square state with the square prismatic symmetry $D^{4\mathrm{h}}$ is a non-axisymmetric solution which emerges first as $\mu$ increases.
Here, $\rho$ is said to be with prismatic dihedral symmetry $D^{2n\mathrm{h}}$ (order $8n$), if in some spherical coordinate $(\theta,\varphi)$, $\rho$ satisfies
\begin{equation}\label{eqn:dnh}
\rho(\theta,\varphi)=\rho(\theta+\pi/n,\varphi)=\rho(-\theta,\varphi)=\rho(\theta,-\varphi).
\end{equation}
For a larger $\mu$, the oblate state becomes a 6-saddle via another pitchfork bifurcation, and a \emph{hexagon} state exists as a 4-saddle with $D^{6\mathrm{h}}$ symmetry.
It is a non-axisymmetric solution with density concentrating at six vertices of a regular hexagon on a great circle.
It is noted that for a two-dimensional Onsager free-energy model with the Onsager potential, a critical point with four (or six) fold symmetry was also obtained with probability density concentrating on four (or six) directions uniformly distributed on the unit circle \cite{wang2008multiple,wang2010phase}.

Finally we show the solution landscape at $\mu=100$ in \cref{fig:ons5}(E), where the isotropic state is a 14-saddle.
The \emph{cubic} state and the \emph{octahedral} state are two non-axisymmetric solutions with cubic symmetry as predicted in \cite{ball2021axisymmetry}.
The cubic state has eight density peaks at vertices of a regular cube, and the octahedral state has six at vertices of a regular octahedron.
Both the cubic state and the octahedral state have energy lower than $E(\rho_0)$.
Besides, the \emph{quarter} state is also a non-axisymmetric solution with $D^{4\mathrm{h}}$ symmetry.
Together with the square and hexagon states, these five non-axisymmetric solutions are plotted in spherical coordinates for better illustration in \cref{fig:ons5}(F).

\begin{figure*}[hbtp]
\centering
\includegraphics[width=\linewidth]{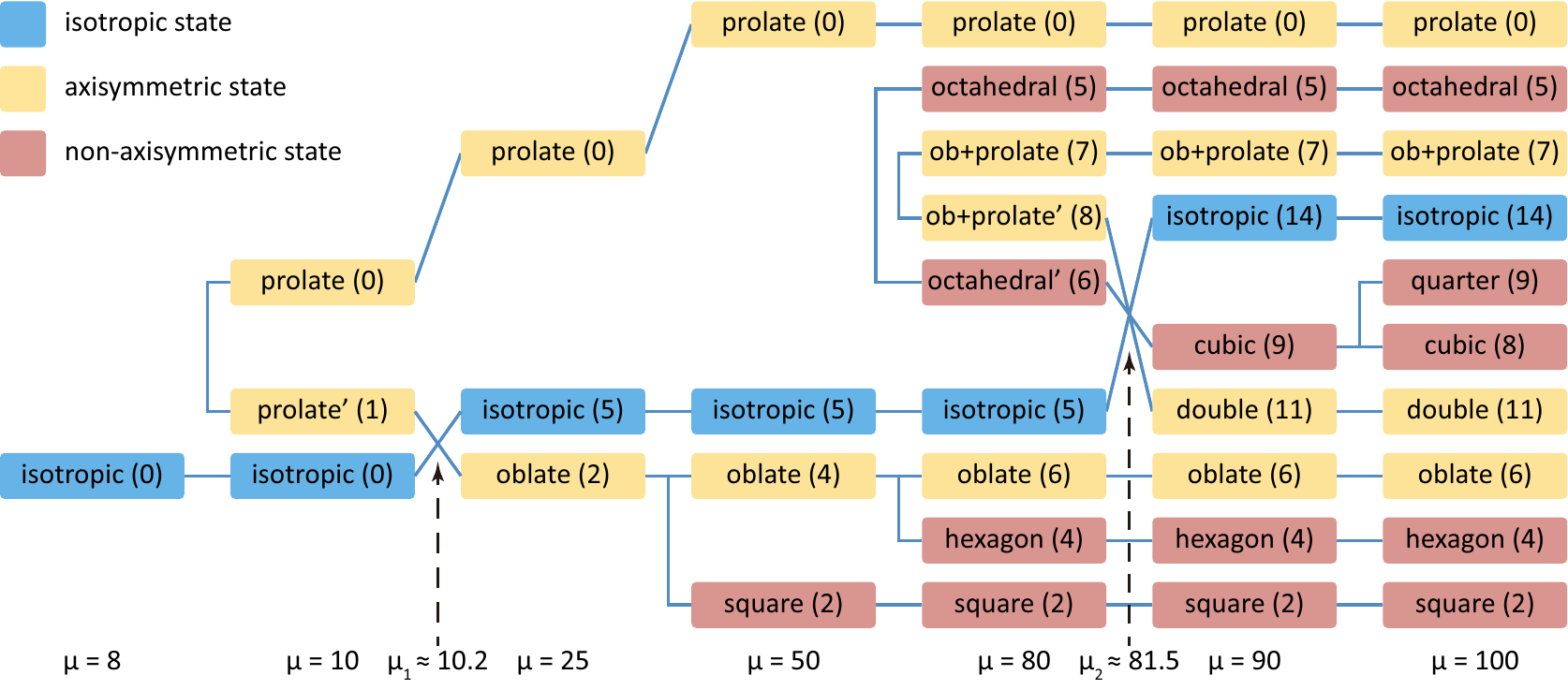}
\caption{The bifurcation diagram for the Onsager potential.
The color of each node indicates its axisymmetric property with its index in the bracket, and the corresponding parameter $\mu$ is labelled below.
Two bifurcation points of the isotropic state are $\mu = \mu_1$ and $\mu = \mu_2$.
}
\label{fig:ons6}
\end{figure*}

To present a full description of the emergence of these critical points, the bifurcation diagram is shown in \cref{fig:ons6}.
At $\mu = 10$, two axisymmetric states have emerged as a minimum (prolate) and a 1-saddle (prolate$'$) via a saddle-node bifurcation.
The first bifurcation of the isotropic state $\rho_0$ is a transcritical bifurcation at $\mu=\mu_1$ where the prolate$'$ 1-saddle becomes an oblate 2-saddle.
Then a secondary bifurcation occurs in the primary branch of the oblate state, leading to the non-axisymmetric square 2-saddle.
Successively, another secondary bifurcation occurs in this primary branch with the appearance of the hexagon 4-saddle.
Similar to the pair of prolate states, two octahedral states and two oblate-prolate (ob+prolate) states also emerge via saddle-node bifurcations (critical points with a higher index are labelled as octahedral$'$ and ob+prolate$'$ in \cref{fig:ons6}).
The isotropic state becomes a 14-saddle after a transcritical bifurcation at $\mu=\mu_2$ when the cubic and double states appear.
Then the non-axisymmetric quarter state is generated from the primary branch of the cubic state via a secondary bifurcation.

For the Onsager potential, our numerical results give positive evidence for the following conjecture.
\begin{conjecture}
  Any local minimizer is axisymmetric.
  For $\mu>\mu_1$, the global minimizer is a prolate state---an axisymmetric critical point with two density peaks at two antipodal points.
\end{conjecture}
Meanwhile, we have the following conjecture for the Onsager potential.
\begin{conjecture}
  There exists a non-axisymmetric critical point with $D^{2n\mathrm{h}}$ symmetry for some $\mu>0$.
\end{conjecture}
More precisely, as $\mu$ increases, we conjecture that successive secondary bifurcations will occur in the oblate branch, leading to critical points with densities equally distributed on $2n$ points of a great circle.
A similar result has been obtained that the process of finding solutions of the two-dimensional Onsager free-energy model with the Onsager potential can also be continued in \cite{wang2008multiple,wang2010phase}.
However, this conjecture in this three-dimensional case is hard to prove using the procedure in \cite{ball2021axisymmetry}, because the prolate state is an axisymmetric state (thus with $D^{2n\mathrm{h}}$ symmetry) with a lower energy.

\begin{figure*}[hbtp]
\centering
\includegraphics[width=\linewidth]{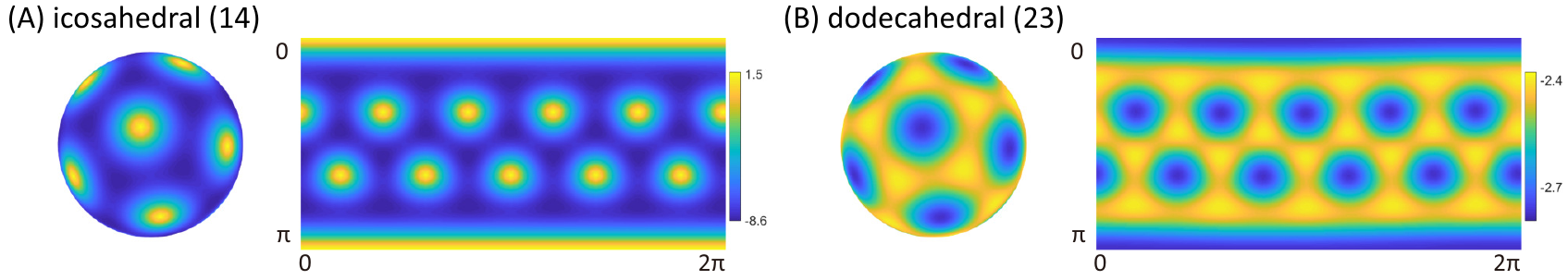}
\caption{Three-dimensional plots and expanded two-dimensional plots of non-axisymmetric critical points with icosahedral symmetry at $\mu=270$ for the Onsager potential.
(A) The icosahedral state. (B) The dodecahedral state.
}
\label{fig:ons7}
\end{figure*}

At $\mu=270>\mu_3$, we numerically found non-axisymmetric critical points with icosahedral symmetry, named an \emph{icosahedral} state and a \emph{dodecahedral} state, as shown in \cref{fig:ons7} in some coordinates.
The isotropic state is a 27-saddle, and the icosahedral state and the dodecahedral state are index-14 and index-23 saddles, respectively.
Here, a probability density $\rho$ has \emph{icosahedral symmetry} provided for $\forall \mathbf{R} \in I^{120}$,
\begin{equation}\label{eqn:icosahedraldef}
  \rho(\mathbf{p}) = \rho(\mathbf{R} \mathbf{p}), \quad \mathrm{a.e.}\, \mathbf{p}\in S^2.
\end{equation}
The icosahedral group $I^{120}$ is the symmetry group of a regular icosahedron with 120 elements, generated by orthogonal transformations,
\begin{equation}\label{eqn:icosahedralgener}
  \mathbf{R}_1=\left(\begin{array}{ccc}
       -1 & 0 & 0 \\
        0 & 1 & 0 \\
        0 & 0 & 1 \\
      \end{array}\right),\,
  \mathbf{R}_2=\frac{1}{2}\left(\begin{array}{ccc}
        -\phi & 1/\phi & 1 \\
        1/\phi & -1 & \phi \\
        1 & \phi & 1/\phi \\
      \end{array}\right),\,
  \mathbf{R}_3=\left(\begin{array}{ccc}
        1 & 0 & 0 \\
        0 & -1 & 0 \\
        0 & 0 & 1 \\
      \end{array}\right),
\end{equation}
where $\phi=(1+\sqrt{5})/2\approx 1.618$ is the golden ratio.
The following theorem shows that non-axisymmetric critical points with icosahedral symmetry exist in the Onsager model with a wide range of kernel potentials.

\begin{theorem}\label{thm:icosahedral}
  There exists a non-axisymmetric critical point of the Onsager model \cref{eqn:onsager} with icosahedral symmetry, if
  \begin{equation}\label{eqn:icosahedralthm}
    \int_{-1}^1 k(t) P_6(t) \mathrm{d}t < -2,
  \end{equation}
  where $P_6(\cdot)$ is the sixth Legendre polynomial.
\end{theorem}
\begin{proof}
Analogous to Proof of Theorem 3.5 in \cite{ball2021axisymmetry}, the Onsager model $E$ can attain a minimizer $\rho_{\mathrm{c}} \in L^1(S^2)$ among the probability density functions with icosahedral symmetry,
\begin{eqnarray}\label{eqn:icosahedralset}
  \mathcal{A}=\bigg\{\rho \in L^1(S^2): \int_{S^2} \rho(\mathbf{p}) \mathrm{d}\mathbf{p} = 1, \rho(\mathbf{p})\geqslant0,  \qquad\qquad\qquad \nonumber \\
  \qquad \rho(\mathbf{p}) = \rho(\mathbf{R} \mathbf{p}),\quad \mathrm{a.e.}\, \mathbf{p}\in S^2, \forall \mathbf{R}\in I^{120} \bigg\},
\end{eqnarray}
and $\rho_{\mathrm{c}}$ is also a critical point of the Onsager model $E$.
If $\rho_{\mathrm{c}}$ is axisymmetric, we would have $\rho_{\mathrm{c}} = \rho_0$ from the icosahedral symmetry.
Therefore, we only prove $\rho_{\mathrm{c}}\neq \rho_0$ if \cref{eqn:icosahedralthm} holds.

An example of a sixth-order (not normalized) spherical harmonic with icosahedral symmetry is \cite{meyer1954symmetries},
\begin{equation}\label{eqn:icosahedralharmonic}
  u(\mathbf{p}) = \left(\phi^2 p_1^2-p_2^2\right)\left(\phi^2 p_2^2-p_3^2\right)\left(\phi^2 p_3^2-p_1^2\right) + \frac{2+\sqrt{5}}{21}.
\end{equation}
It can be verified by direct calculations that $\Delta u(\mathbf{p})=-42u(\mathbf{p})$, $\int_{S^2} u(\mathbf{p}) \mathrm{d}\mathbf{p}=0$ and $u(\mathbf{p})=u(\mathbf{R}_i\mathbf{p})$ for $i=1,2,3$.
The Funk--Hecke Theorem \cite{atkinson2012spherical} implies that,
\begin{equation}\label{eqn:icosahedralfunkhecke}
  \int_{S^2} k(\mathbf{p}\cdot \mathbf{q}) u(\mathbf{q}) \mathrm{d} \mathbf{q}
  = 2\pi u(\mathbf{p}) \int_{-1}^{1} k(t) P_6(t) \mathrm{d}t, \quad \mathbf{p}\in S^2,
\end{equation}
where $P_6(\cdot)$ is the sixth Legendre polynomial.
From
\begin{eqnarray}\label{eqn:icosahedral2v}
  \frac{\delta^2 E}{\delta \rho^2}(\rho_0)[u]&=
  \displaystyle\int_{S^2} \left[4\pi u^2(\mathbf{p}) + u(\mathbf{p}) \int_{S^2} k(\mathbf{p} \cdot\mathbf{q}) u (\mathbf{q}) \mathrm{d}\mathbf{q} \right]\mathrm{d}\mathbf{p}\nonumber\\
  &= 2\pi \displaystyle\int_{S^2}u^2(\mathbf{p})\mathrm{d}\mathbf{p}
  \left(2+\int_{-1}^{1} k(t) P_6(t) \mathrm{d}t \right),
\end{eqnarray}
$\rho_0$ is not a local minimizer of $\mathcal{A}$ if \cref{eqn:icosahedralthm} holds, and consequently $\rho_{\mathrm{c}} \neq \rho_0$.
\end{proof}

For the Onsager potential, the condition \cref{eqn:icosahedralthm} is equivalent to $\mu>\mu_3$, which is consistent with our numerical results.
In fact, there exist non-constant spherical harmonics with icosahedral symmetry of order $r\in R_{\mathrm{I}}=\{ 6,10,12,15,16,18,20,21,22,24$--$28,\cdots\}$ \cite{zheng2000explicit}.
Therefore, by taking $u$ as an $r$-th-order spherical harmonic with icosahedral symmetry, it can be proved that there exists a non-axisymmetric critical point of the Onsager model \cref{eqn:onsager} with icosahedral symmetry, if
\begin{equation}\label{eqn:icosahedral4}
    \int_{-1}^1 k(t) P_r(t) \mathrm{d}t < -2,
\end{equation}
for some $r\in R_{\mathrm{I}}$ where $P_r(\cdot)$ is the $r$-th Legendre polynomial.

\section{Conclusions and discussions}\label{sec:conclusions}

In this article, we studied the solution landscapes of the Onsager model which involves the effects of nonlocal molecular interactions \cite{d2020numerical}.
The solution landscape provides a systematical numerical approach to find multiple critical points and reveal a global structure of the Onsager model. Special attention is paid to identify non-axisymmetric critical points, which are difficult to obtain analytically and numerically.
We first computed the solution landscapes for the dipolar potential and the Maier--Saupe potential and showed all critical points are axisymmetric.
Next, we constructed the solution landscapes for the coupled dipolar/Maier--Saupe potential and identified a novel non-axisymmetric critical point---the \emph{tennis} state.
The solution bifurcations and the existence region of the tennis state are discussed based on the numerical results.
Finally, the solution landscapes are investigated for the Onsager potential.
We found multiple non-axisymmetric critical points, including \emph{square, hexagon, octahedral, cubic, quarter, icosahedral}, and \emph{dodecahedral} states.

The primary bifurcations of the isotropic state $\rho_0$ for the Onsager potential have attracted much attention of analytical studies \cite{ball2021axisymmetry,vollmer2017critical}.
Here, the bifurcation diagram of critical points for the Onsager potential is computed based on the solution landscape in order to show a full description of the emergence of the critical points.
First, the bifurcation of the isotropic state at $\mu=\mu_1$ is shown to lead to an axisymmetric critical point branch, which is consistent with the analytical results in \cite{vollmer2017critical}.
It has been conjectured that the other bifurcations of the isotropic state could lead to non-axisymmetric critical points \cite{ball2021axisymmetry}.
The existence of the cubic state at $\mu>\mu_2$ and the dodecahedral state at $\mu>\mu_3$ accords with this conjecture.
Furthermore, we believe the larger bifurcation points of the isotropic state can also lead to non-axisymmetric critical points.
It should be pointed out that the kernel $k(\mathbf{p} \cdot\mathbf{q})$ in \cref{eqn:onsager}, which is proportional to the number density and the average excluded volume \cite{ball2017liquid,palffy2017onsager,taylor2021cavity}, cannot become unconditionally large.
Therefore, an overlarge value of $\mu$ can be unphysical, and the results for larger parameters are not presented here.

An open question of the Onsager functional is whether local or global minimizers for the Onsager potential are axisymmetric \cite{ball2021axisymmetry}.
For instance, there could be secondary bifurcation from a primary branch of axisymmetric solutions to a non-axisymmetric minimizer.
We systematically constructed the solution landscapes of the Onsager functional with the Onsager potential for $\mu \leqslant100$ and demonstrated that secondary bifurcations indeed occur and lead to non-axisymmetric saddle points, not minimizers.
Therefore, our numerical results support the conjecture that all local minimizers are axisymmetric for the Onsager potential.

\section*{Acknowledgments}
This work was supported by the National Natural Science Foundation of China No.~12050002, No.~21790340.
The authors thank Prof.~John Ball and Prof.~Wei Wang for helpful discussions.
J.~Y. acknowledges the support from the Elite Program of Computational and Applied Mathematics for Ph.D. Candidates of Peking University.

\appendix
\section{Reduced Onsager model}\label{apd:a}

Generally speaking, the Onsager free-energy model \cref{eqn:onsager} is an infinite-dimensional problem.
However, for the coupled dipolar/Maier--Saupe potential (including the dipolar potential and the Maier--Saupe potential), the Onsager model \cref{eqn:onsager} can be reduced to a finite-dimensional model.
For the Onsager potential, this reduced model cannot work.

Supposing that $\rho$ is a critical point for the coupled dipolar/Maier--Saupe potential, we rewrite the Euler--Lagrange equation \cref{eqn:onsagerel} as,
\begin{equation}\label{eqn:onsagerelphi}
  \log \phi(\mathbf{p}):=\log\rho(\mathbf{p})-\xi=-\int_{S^2} k(\mathbf{p} \cdot\mathbf{q}) \rho (\mathbf{q}) \mathrm{d}\mathbf{q}.
\end{equation}
The dipolar potential $-\sigma (\mathbf{p}\cdot\mathbf{q})$ and the Maier--Saupe potential $-\kappa (\mathbf{p}\cdot\mathbf{q})^2$ are eigenfunctions of the Laplace--Beltrami operator $\Delta$ with respect to $\mathbf{p}$ on $S^2$, and the corresponding eigenvalues are $-2$ and $-6$ \cite{fatkullin2005critical}, so we have,
\begin{equation}\label{eqn:laplacephi}
  (\Delta+2)(\Delta+6) \log \phi(\mathbf{p})=-\int_{S^2} (\Delta+2)(\Delta+6)k(\mathbf{p} \cdot\mathbf{q}) \rho (\mathbf{q}) \mathrm{d}\mathbf{q}=0.
\end{equation}
Since the corresponding eigenspaces $\Lambda_1$ and $\Lambda_2$ are three- and five-dimensional subspaces spanned by spherical harmonics of the first and second orders respectively, $\phi(\mathbf{p})$ in \cref{eqn:onsagerelphi} can be expressed as a simple form of,
\begin{equation}\label{eqn:fmphi8}
  \exp( m_1p_1 + m_2p_2 + m_3p_3 + m_4(p_1^2-p_3^2)+ m_5(p_2^2-p_3^2)
  +m_6 p_1p_2 + m_7 p_2p_3 +m_8p_1p_3).
\end{equation}
Furthermore, by properly choosing the coordinate system to diagonalize the second-order term in \cref{eqn:fmphi8}, $\phi$ can be expressed as
\begin{equation}\label{eqn:fmphi5}
   \phi_{\mathrm{r}}(\mathbf{p};\bm m) = \exp\left(m_1p_1 + m_2p_2 + m_3p_3 + m_4(p_1^2-p_3^2)+ m_5(p_2^2-p_3^2)\right),
\end{equation}
with only five parameters.
Therefore, we consider a reduced finite-dimensional Onsager model,
\begin{equation}\label{eqn:onsagerreduced}
   E_{\mathrm{r}}(\bm m) = E\left(\rho_{\mathrm{r}}(\bm m)\right), \quad
   \rho_{\mathrm{r}}(\mathbf{p};\bm m)=Z_{\mathrm{r}}(\bm m)^{-1} \phi_{\mathrm{r}}(\mathbf{p};\bm m),
\end{equation}
where $Z_{\mathrm{r}}(\bm m) = \int_{S^2} \phi_{\mathrm{r}}(\mathbf{p};\bm m) \mathrm{d}\mathbf{p}$ is the partition function of $\phi_{\mathrm{r}}$.
The gradient of the reduced model \cref{eqn:onsagerreduced} is
\begin{equation}\label{eqn:fmene_m}
  \frac{\partial E_{\mathrm{r}}}{\partial\bm m} = \int_{S^2}
  \frac{\partial\rho_{\mathrm{r}}}{\partial\bm m}\frac{\delta E}{\delta \rho}(\rho_{\mathrm{r}})  \mathrm{d}\mathbf{p},
\end{equation}
where
\begin{equation}\label{eqn:fmrho_m}
  \frac{\partial\rho_{\mathrm{r}}}{\partial\bm m} =
  \frac{1}{Z_{\mathrm{r}}}\frac{\partial\phi_{\mathrm{r}}}{\partial\bm m}
  + \frac{\rho_{\mathrm{r}}}{Z_{\mathrm{r}}}\frac{\partial Z_{\mathrm{r}}}{\partial\bm m}.
\end{equation}
Therefore, any critical point of the Onsager model \cref{eqn:onsager} can also be found as a critical point (with an orthogonal transformation) in the reduced Onsager model \cref{eqn:onsagerreduced}.
However, there may exist extraneous solutions in the reduced Onsager model \cref{eqn:onsagerreduced}, so each critical points of \cref{eqn:onsagerreduced} should be verified.
The Hessian of \cref{eqn:onsagerreduced} is
\begin{equation}\label{eqn:fmene_mv}
  \frac{\partial^2 E_{\mathrm{r}}}{\partial\bm m^2} [\bm v] =
  \int_{S^2}
  \frac{\partial^2 \rho_{\mathrm{r}}}{\partial\bm m^2}[\bm v]\frac{\delta E}{\delta \rho}(\rho_{\mathrm{r}})
  +\frac{\partial\rho_{\mathrm{r}}}{\partial\bm m} \frac{\delta^2 E}{\delta \rho^2} (\rho_{\mathrm{r}}) \left[ \frac{\partial\rho_{\mathrm{r}}}{\partial\bm m}\cdot \bm v\right]
  \mathrm{d}\mathbf{p},
\end{equation}
where
\begin{equation}\label{eqn:fmrho_mm}
  \frac{\partial^2 \rho_{\mathrm{r}}}{\partial\bm m^2} =
  \frac{1}{Z_{\mathrm{r}}}\frac{\partial^2\phi_{\mathrm{r}}}{\partial\bm m^2}
  -\frac{\rho_{\mathrm{r}}}{Z_{\mathrm{r}}}\frac{\partial^2 Z_{\mathrm{r}}}{\partial\bm m^2}
  +\frac{2\rho_{\mathrm{r}}}{Z_{\mathrm{r}}^2} \frac{\partial Z_{\mathrm{r}}}{\partial\bm m}
  \frac{\partial Z_{\mathrm{r}}}{\partial\bm m}^\top
  -\frac{1}{Z_{\mathrm{r}}^2} \left(\frac{\partial\phi_{\mathrm{r}}}{\partial\bm m}\frac{\partial Z_{\mathrm{r}}}{\partial\bm m}^\top
  +\frac{\partial Z_{\mathrm{r}}}{\partial\bm m}\frac{\partial\phi_{\mathrm{r}}}{\partial\bm m}^\top\right).
\end{equation}
It can be easily observed that for a critical point of both models, its index in the Onsager model \cref{eqn:onsager} is no less than that in \cref{eqn:onsagerreduced}.

\bibliographystyle{elsarticle-num}
\bibliography{onsager}

\end{document}